\documentclass[10pt]{amsart}
\usepackage{amssymb}
\usepackage{dsfont}
\usepackage{amscd}
\usepackage{mathabx}
\usepackage[mathscr]{euscript} % for the power-set P
\usepackage[all]{xy}
\usepackage{url}
\usepackage{stmaryrd}
\usepackage{comment}
\usepackage[retainorgcmds]{IEEEtrantools}
\usepackage{hyperref}

\newtheorem{theorem}{Theorem}[section]
\newtheorem{lemma}[theorem]{Lemma}

\newtheorem{proposition}[theorem]{Proposition}

\newtheorem{corollary}[theorem]{Corollary}

\theoremstyle{definition}
\newtheorem{definition}[theorem]{Definition}

\newtheorem{remark}[theorem]{Remark}
\newtheorem{question}[theorem]{Question}

\newtheorem{assumption}[theorem]{Assumption}

\newcommand{\pars}{\par\smallskip}
\newcommand{\parm}{\par\medskip}

\newcommand{\Zz}{{\mathds{Z}}}

\newcommand{\Cc}{{\mathds{C}}}
\newcommand{\Rr}{{\mathds{R}}}

\newcommand{\Qq}{{\mathds{Q}}}

\newcommand{\Uu}{{\mathds{U}}}

\newcommand{\ga}{\mathbb{G}_{\rm{a}}}

\title[Generic and supertight automorphisms]{On generic and supertight automorphisms}

\author[P. KOWALSKI]{Piotr Kowalski$^{\diamondsuit}$}
\thanks{$^{\diamondsuit}$
 Supported by the Narodowe Centrum Nauki grant no. 2021/43/B/ST1/00405 and by the T\"{u}bitak grant no. 1001-124F359.}
\address{$^{\diamondsuit}$Instytut Matematyczny\\
Uniwersytet Wroc{\l}awski\\
Wroc{\l}aw\\
Poland}
\email{pkowa@math.uni.wroc.pl} \urladdr{http://www.math.uni.wroc.pl/\textasciitilde pkowa/ }

\author[P. U\v{g}urlu Kowalski]{P\i nar U\u{g}urlu Kowalski$^{\heartsuit}$}
\address{$^{\heartsuit}$Istanbul Bilgi University}
\email{pinar.ugurlu@bilgi.edu.tr}


\DeclareMathOperator{\gal}{Gal}

\DeclareMathOperator{\id}{id}
\DeclareMathOperator{\acl}{acl}\DeclareMathOperator{\aut}{Aut}
\DeclareMathOperator{\thh}{Th}\DeclareMathOperator{\cld}{cld}
\DeclareMathOperator{\tp}{tp}
\DeclareMathOperator{\fix}{Fix}\DeclareMathOperator{\dcl}{dcl}
\DeclareMathOperator{\alg}{alg}

\begin{document}

\begin{abstract}
We show that generic automorphisms of stable groups are supertight (a notion introduced in \cite{pinulla}) in a strong sense. In particular, we obtain the existence of supertight automorphisms. We also answer a question from \cite{pinulla} concerning the relationship between supertight automorphisms of $\mathrm{PGL}_2(K)$ and generic automorphisms of the underlying field $K$. Moreover, we provide partial evidence---already suggested by Hrushovski in \cite{Hrupfstr}---toward the principle that ``fixed points are pseudofinite’’ in the setting of generic automorphisms of simple groups of finite Morley rank.
\end{abstract}

\maketitle

\section{Introduction}
The Cherlin-Zilber conjecture (proposed independently by Cherlin \cite{cherczc} and Zilber \cite{zilzcc}) states that infinite simple groups satisfying some natural dimension conditions (that is, they have \emph{finite Morley rank}) must be algebraic. This note concerns a possible approach, suggested by Hrushovski in \cite{Hrupfstr}, towards this conjecture. A more general conjecture of Zilber---the Trichotomy Conjecture---asserting the algebraicity of certain one-dimensional sets (without any assumed group structure, see \cite{zilbericm}) was  refuted by Hrushovski in \cite{HR2}. Thus, the Cherlin–Zilber conjecture perhaps remains “closest in spirit’’ to Zilber’s Trichotomy Conjecture. The Cherlin-Zilber conjecture is still open; however, the three-dimensional case was settled positively by Fr\'{e}con in \cite{frecon3}.

It should be mentioned that the most promising strategy for addressing the Cherlin-Zilber conjecture is the so-called \emph{Borovik program}, which adapts techniques from finite group theory to classify infinite simple groups of finite Morley rank. This framework divides such groups into even, odd, mixed, and degenerate types. Significant progress has been made: in particular, mixed types have been shown not to exist, and even-type groups are isomorphic the groups of rational points of algebraic groups over algebraically closed fields of characteristic $2$ (see \cite{abc}).

Hrushovski's suggestion in the paragraph titled `Connection to the Borovik program?' at the end of \cite{Hrupfstr} can be summarized as follows.
\begin{enumerate}
  \item Take a generic automorphism $\sigma$ of a simple group of finite Morley rank $G$.

  \item Show that the fixed-point subgroup $\fix(\sigma)$ is a pseudofinite group, i.e., a model of the common theory of all finite groups. (Note, however; that $\fix(\sigma)$ \emph{need not} be simple as an abstract group; see Remark \ref{tricky}(1).)

  \item Show that $\fix(\sigma)\cong H(F)$, where $H$ is a simple algebraic group defined over a pseudofinite field $F$.

  \item Conclude that $G\cong H(K)$, where $K$ is an algebraically closed field extending $F$.
\end{enumerate}
Hrushovski also points out that even without achieving the crucial Item (2) above:
$$\text{``counting arguments applied in finite group theory might be applied via $\mu$''}.$$
Here, $\mu$ is a measure on $\fix(\sigma)$ constructed in \cite[Proposition 11.1]{Hrupfstr} that resembles the counting measure on a pseudofinite structure. This suggestion has not yet been pursued; we provide further comments in Section \ref{secpf}.

Clearly, the primary obstacle here is the pseudofiniteness of the subgroup of invariants of the generic automorphism $\sigma$. One might still hope to show that the Cherlin–Zilber conjecture is \emph{equivalent} to pseudofiniteness as described above; such a partial result (see \cite[Conjecture 1.5]{pinulla}) would represent a significant step toward justifying Hrushovski’s suggestion. In \cite{pinulla}, the notion of supertight automorphisms is considered (see Definition \ref{defstight}) and the following two steps are specified.

\begin{itemize}
  \item The \emph{Algebraic identification step} corresponding to Item (4) above.

  \item The \emph{Model-theoretic step}, which consists of showing that a generic automorphism is supertight.
\end{itemize}
Afterwards, fixed-point subgroups of supertight automorphisms were studied by Karhum\"{a}ki in \cite{ulla2}.

In this paper, we establish the aforementioned ``model-theoretic step''; specifically, we prove that generic automorphisms are supertight in a strong sense (see Theorem~\ref{gensup}). As a consequence, we also show that supertight automorphisms \emph{exist} on elementary extensions (see Corollary~\ref{exist}). Furthermore, we address a question from \cite{pinulla} regarding the connection between supertight automorphisms of $\mathrm{PGL}_2(K)$ and generic automorphisms of the field $K$ (see Section~\ref{secalggps}). We also show that the fixed-point groups satisfy the Ax conditions for pseudofiniteness, assuming a very weak version of elimination of imaginaries (see Section~\ref{secpf}). Whether this is related to actual pseudofiniteness for simple groups of finite Morley rank remains the main challenge.

This paper is organised as follows. In Section \ref{secgeneric}, we discuss properties of generic automorphisms on stable structures. In Section \ref{sectands}, we compare generic automorphisms on stable groups to supertight automorphisms and show the results mentioned above. In Section \ref{secaxiom}, we discuss an axiomatization of generic automorphisms using the results existing in the literature. In Section \ref{secalggps}, we answer
  \cite[Question 6.1]{pinulla}. Finally, in Section \ref{secpf}, we provide a partial evidence (as suggested by Hrushovski in \cite{Hrupfstr}) toward the pseudofiniteness of the group of fixed points of a generic automorphism.

\section{Generic automorphisms}\label{secgeneric}
Let $L$ be a language and $M$ be an $L$-structure. For $\sigma\in \aut(M)$, we want to say when $\sigma$ is generic. Our definition essentially comes from \cite{KiPi}. Let $T:=\thh(M)$ and $T'$ be the ``'Morleyisation'' of $T$, see \cite[page 62]{HoMo} (Hodges actually suggested the name ``atomization''), obtained by naming each $L$-formula by a new relation symbol in the new language $L'$ expanding $L$. Then, the models of $T'$ are exactly the $L$-structures elementarily equivalent with $M$ and $T'$ has quantifier elimination (in $L'$). In particular, the theory $T'$ is model complete. Let $L'_{\sigma}:=L'\cup \{\sigma\}$, where $\sigma$ is a new unary function symbol, and
$$T'_{\sigma}:=T'\cup \{\text{$\sigma$ is an $L'$-automorphism}\}.$$
\begin{definition}[page 263 in \cite{KiPi}]\label{defgen}
We call an automorphism $\sigma$ as above \emph{generic}, if $(M,\sigma)$ is an existentially closed model of $L'_{\sigma}$.
\end{definition}
It is clear that the $L'_{\sigma}$-theory $T'_{\sigma}$ is inductive, so the following is obvious.
\begin{lemma}\label{ec}
If $M$ and $\sigma$ are as above, then there is an elementary extension $M\preccurlyeq M'$  and $\sigma'\in \aut(M')$ extending $\sigma$ such that $\sigma'$ is generic.
\end{lemma}

\begin{remark}\label{present}
We comment here on the definition of a generic automorphism.
\begin{enumerate}
  \item Usually (see \cite{BGH, ChPi, KiPi}), one begins with a complete theory admitting quantifier elimination and defines the notion of a generic automorphism relative to it, as in \cite[p.~263]{KiPi}. However, since we are interested in groups of finite Morley rank, we prefer to start with an arbitrary structure. Both approaches clearly yield the same notion of genericity. In particular, if $M$ is an algebraically closed field, then $\sigma\in \aut(M)$ is generic if and only if $(M,\sigma)$ is a model of the theory ACFA (see \cite{acfa1}).

  \item There is also a distinct notion of a generic automorphism of a countable homogenous structure (see \cite{trussgen}). In contrast to the automorphisms described in Lemma~\ref{ec}, these \emph{need not} exist. These two notions of genericity are compared in \cite{barzam}.

\end{enumerate}
\end{remark}

\begin{assumption}\label{assstable}
For the remainder of the paper, we assume that \textbf{the structure $M$ is stable}; that is, the theory $T$ (or equivalently, $T'$) above is stable. We also fix a monster model $\Uu\models T'$, which is a sufficiently saturated elementary extension of $M$. All other models of $T$ considered are tacitly assumed to be elementary substructures of $\Uu$ of small cardinality relative to the saturation of $\Uu$.
\end{assumption}
The following is quite obvious and should be folklore.
\begin{lemma}\label{tensor}
Suppose that $k>0$ and for each $i\in \{1,\ldots,k\}$ we have
$$M\prec N_i\prec \Uu,\ \ M\prec N_i'\prec \Uu$$
such that $(N_1,\ldots,N_k)$ and $(N_1',\ldots,N_k')$ are forking independent tuples over $M$. Suppose also that $\sigma\in \aut(M)$ and for each  $i\in \{1,\ldots,k\}$ we have
$$\sigma_i:N_i \xrightarrow{\cong} N_i'$$
extending $\sigma$. Then, there is $\tau\in \aut(\Uu)$ extending all $\sigma_1,\ldots,\sigma_k$.
\end{lemma}
\begin{proof}
We need to show that the map
$$\sigma_1\cup \ldots \cup \sigma_k:N_1\cup \ldots \cup N_{k} \to N_1'\cup \ldots \cup N_{k}'$$
is elementary. For convenience, we do it only in the case of $k=2$. Since $N_1'=\sigma_1(N_1)$ and $N_2'=\sigma_2(N_2)$, it is enough to show the following.
\begin{equation}
  \tag{$*$}
\sigma_1\left(\tp\left(N_2/N_1\right)\right)=\tp\left(N_2'/N_1'\right).
\end{equation}
Clearly, $\tp\left(N_2'/N_1'\right)$ is an extension of $\tp\left(N_2'/M\right)$ which is non-forking since $N_2'$ is forking-independent from $N_1'$ over $M$. We also have
$$\sigma\left(\tp\left(N_2/M\right)\right)=\tp\left(N_2'/M\right),$$
since $N_2'=\sigma_2(N_2)$ and $\sigma_2$ extends $\sigma$. Since $\sigma_1$ extends $\sigma$, we obtain that $\sigma_1\left(\tp\left(N_2/N_1\right)\right)$ is an extension of $\tp\left(N_2'/M\right)$ as well. This last extension is non-forking, since it is the image by $\sigma_1$ of the non-forking extension $\tp(N_2/M)\subseteq \tp(N_2/N_1)$ (we use here that $N_2$ is forking-independent from $N_1$ over $M$). By stationarity of types over models in stable theories, we obtain $(*)$.
\end{proof}
The following result and its proof are essentially taken from \cite[Lemma 1.12]{acfa1}, where it was done for the case of algebraically closed fields.
\begin{theorem}\label{kgen}
Suppose that $\sigma\in \aut(M)$ is generic and $1\leqslant k<l$. Then the following holds.
\begin{enumerate}
  \item The automorphism $\sigma^k$ is generic as well.

  \item If $k\mid l$, then $\fix(\sigma^k)\lneqq \fix(\sigma^l)$.
\end{enumerate}

\end{theorem}
\begin{proof}
The following result is crucial for the proof of Item (1). Its proof is adapted from that of \cite[Lemma 1.12]{acfa1}; we note that the genericity of $\sigma$ is not required here.

\medskip
\textbf{Claim}
\\
Suppose that $M\preccurlyeq N\prec \Uu$ and $\tau\in \aut(N)$ such that $\tau$ extends $\sigma^k$. Then, there is $\rho\in \aut(\Uu)$ such that $\rho$ extends $\sigma$ and $\rho^k$ extends $\tau$
\begin{proof}[Proof of Claim]
Let $p:=\tp(N/M)$. For any $i\in \{1,\ldots,k-1\}$ let $N_i$ realize the unique non-forking extension of $\sigma^i(p)$ to $N_0\cup \ldots \cup N_{i-1}$ where $N_0:=N$. Let
$$\sigma_i:N_{i-1}\xrightarrow{\cong} N_i$$
extend $\sigma$ for $i$ as above and we define
$$\sigma_k:=\tau\circ \sigma_{1}^{-1}\circ \ldots \circ \sigma_{k-1}^{-1}:N_{k-1}\xrightarrow{\cong} N_0=N.$$
Note that $\sigma_k$ extends $\sigma$ as well, since $\tau$ extends $\sigma^k$ and the product of the inverses extends $\sigma^{-(k-1)}$. We now have the following:
\begin{itemize}
  \item $N_0,\ldots,N_{k-1}$ is a sequence of models of $T'$ which is forking independent over $M$;
  \item the maps
  $$\sigma_1:N_{0}\to N_1, \sigma_2:N_{1}\to N_2, \ldots,\sigma_k:N_{k-1}\to N_0$$
  are isomorphisms extending $\sigma$.
\end{itemize}
By Lemma \ref{tensor} (for $(N_0,\ldots,N_{k-1})$ and $(N_1,\ldots,N_{k-1},N_0)$), there is $\rho\in \aut(\Uu)$ such that for each $i\in \{1,\ldots,k\}$, $\rho$ extends $\sigma_i$. Therefore, $\rho$ extends $\sigma$ and $\rho^k$ extends $\sigma_k\circ \sigma_{k-1}\circ \ldots \circ \sigma_1=\tau$ as desired.
\end{proof}

It is standard now to show that $\sigma^k$ is generic. Let us take a quantifier free $L'_{\sigma}$-formula $\varphi(x)$ over $M$ and an $L'_{\sigma}$-extension $(M,\sigma^k)\subseteq (N,\tau)$ such that $(N,\tau)\models \exists x\varphi(x)$. Note that there is a corresponding $L'_{\sigma}$-formula $\widetilde{\varphi}(x)$ over $M$ such that for any $L'_{\sigma}$-extension $(M,\sigma)\subseteq (M',\sigma')$ we have that
\begin{equation}
  \tag{$**$}
(M',\sigma')\models \exists x\widetilde{\varphi}(x)\ \ \ \ \ \iff\ \ \ \ \ \ \ (M',(\sigma')^k)\models \exists x \varphi(x).
\label{star}
\end{equation}
We take now $\rho\in \aut(\Uu)$ given by Claim. Since $(N,\tau)\subseteq (\Uu,\rho^k)$ is an $L'_{\sigma}$-extension and $\varphi(x)$ is quantifier-free, we get that $(\Uu,\rho^k)\models \exists x\varphi(x)$. By (\ref{star}), we get $(\Uu,\rho)\models \exists x\widetilde{\varphi}(x)$. Since $(M,\sigma)$ is existentially closed, we get that $(M,\sigma)\models \exists x\widetilde{\varphi}(x)$. By (\ref{star}) again, we finally obtain $(M,\sigma^k)\models \exists x \varphi(x)$, which finishes the proof of Item (1).

\medskip

For the proof of Item (2), we proceed similarly. We fix a non-algebraic 1-type $p$ over $M$ which does not fork over $\emptyset$. Let $a_1,a_2,\ldots,a_l$ be a (short) Morley sequence in $p$. Then $a_2,a_3,\ldots,a_{l},a_1$ is a Morley sequence in $p$ as well, hence we get the following
$$(a_1,a_2,\ldots,a_{l-1},a_l)\equiv_M (a_2,a_3,\ldots,a_{l},a_1)\models p^{\otimes l}.$$
Since $p$ does not fork over $\emptyset$, $p$ is $\aut(M)$-invariant. In particular, we have $\sigma(p^{\otimes l})=p^{\otimes l}$.
Hence, there is $\tau\in \aut(\Uu)$ extending $\sigma$ such that
$$\tau(a_1)=a_2,\tau(a_2)=a_3,\ldots,\tau(a_{l-1})=a_l,\tau(a_l)=a_1.$$
Therefore, we get:
$$a_1\in \fix\left(\tau^l\right)\setminus \fix\left(\tau^k\right).$$
Since $(M,\sigma)$ is existentially closed, we obtain $\fix(\sigma^k)\neq \fix(\sigma^l)$.
\end{proof}

\begin{remark}
Given our focus on groups of finite Morley rank, the assumption of stability is more than sufficient for our purposes. Nevertheless, we are not aware of any counterexamples to Theorem~\ref{kgen} in the absence of stability. It is plausible that the arguments presented in this section could be extended to certain ``neo-stability'' contexts.
\end{remark}
We finish this section with one observation and its standard proof is left to the reader. For the notion of bi-interpretability, we refer to \cite[Chapter 5]{HoMo}.
\begin{proposition}\label{biint}
Assume that the structures $\mathcal{M}$ and $\mathcal{N}$ are bi-interpretable and
$$\Gamma:\aut(\mathcal{M})\xrightarrow{\cong} \aut(\mathcal{N})$$
is the natural isomorphism coming from the associated bi-interpretability functor. Then, $\Gamma$ takes generic automorphisms of $\mathcal{M}$ onto generic automorphisms of $\mathcal{N}$.
\end{proposition}

\section{Tight and supertight automorphisms}\label{sectands}

We assume in this section that $M=G$ is a stable group. The definition of tightness from \cite[Definition 4.1]{pinulla}, originally introduced for groups of finite Morley rank, generalizes naturally to the present context. Throughout this paper, ``definable'' means ``definable with parameters''; when we wish to specify a parameter set $A$, we often say ``$A$-definable''.
\begin{definition}\label{deftight}
We say that $\sigma\in \aut(G)$ is \emph{tight}, if for any connected definable subgroup $H\leqslant G$ such that $\sigma(H)=H$, we have that $\fix(\sigma|_H)$ is not contained in any proper definable subgroup of $H$.
\end{definition}
\begin{remark}\label{remtight}
We comment here on the definition of tightness.
\begin{enumerate}
\item Since $\omega$-stable groups satisfy the descending chain condition (dcc) on definable subgroups, for any $A \subseteq G$ there exists a smallest definable subgroup containing $A$, which we denote by  $\cld(A)$. According to \cite[Definition 4.1]{pinulla}, $\sigma\in \aut(G)$ is said to be tight if, for any connected definable subgroup $H\leqslant G$  such that $\sigma(H)=H$, we have:
$$\cld\left(\{g\in H\ |\ \sigma(g)=g\}\right)=H.$$
That is, $\fix(\sigma|_H)$ must be ``cld-dense'' in $H$. Clearly, in the context of $\omega$-stable groups this is just another way of phrasing Definition \ref{deftight}.

\item If $G=H(K)$, where $K$ is an algebraically closed field and $H$ is an algebraic group over $K$, then a subgroup of $G$ is ``cld-dense'' in $G$ if and only if it is Zariski dense in $G$.

\item It is worth noting that under Definition \ref{deftight}, the identity automorphism on $G$ is tight. This may be viewed as a weakness of this particular definition of tightness.
\end{enumerate}
\end{remark}
The notion of tightness can be strengthened in the style of $\sigma$-varieties (see e.g. \cite{KP4}).
\begin{definition}\label{sharp}
We say that $\sigma\in \aut(G)$ is \emph{tight$^{\sharp}$}, if for any connected group $H$ which is interpretable in $G$ and any interpretable epimorphism $f:H\to \sigma(H)$, we have that $(H,f)^{\sharp}$ is not contained in any proper definable subgroup of $H$, where $$(H,f)^{\sharp}:=\{g\in H\ |\ \sigma(g)=f(g)\}.$$
\end{definition}
\begin{remark}\label{remsharp}
We comment here on the definition of tightness$^{\sharp}$.
\begin{enumerate}
\item If $\sigma$ is tight$^{\sharp}$, then it is also tight, as we may take $f=\id_H$ for any definable subgroup $H\leqslant G$  such that $\sigma(H)=H$. Thus, the definition of tightness$^{\sharp}$ significantly relaxes the restrictions found in the original definition of tightness.
    For instance, $H$ no longer needs to be definable (over $\fix(\sigma)$) anymore; it may be merely interpretable. Furthermore, the identity map on $H$ can be replaced by any interpretable epimorphism $f: H \to \sigma(H)$.

\item Under this new definition, the identity map on $G$ is no longer tight$^{\sharp}$. More generally, any definable automorphism of $G$ fails to be tight$^{\sharp}$, except in the case where $G$ is a vector space over $\mathbb{F}_2$.

\item In the context of $\omega$-stable theories, Definition \ref{sharp} translates as ``the $\sharp$-points are cld-dense'' and it is motivated by the properties of the ``$\sharp$-points'' of $\sigma$-varieties (see \cite{KP4} and \cite{kammos}).
\end{enumerate}
\end{remark}
The next result says that generic automorphism from Section \ref{secgeneric} satisfy the stronger form of tightness from Definition \ref{sharp}.
\begin{theorem}\label{gentight}
Any generic automorphism of $G$ is tight$^{\sharp}$.
\end{theorem}
\begin{proof}
Let $G$ and $\sigma$ be as above and assume that $\sigma$ is generic. We take $\Uu$ from Assumption \ref{assstable} for $G=M$.

Suppose now that $\sigma$ is not tight$^{\sharp}$ which is witnessed by $(H,f)$ and an interpretable $H_0\lneqq H$ such that
$$(H,f)^{\sharp}\subseteq H_0.$$
We will reach a contradiction. Let $p_H$ be the generic type of $H$. Since $f:H\to \sigma(H)$ is a definable epimorphism, it takes generic (in the sense of the theory of stable groups) subsets of $H$ to generic subsets of $\sigma(H)$. Therefore, we have:
$$f\left(p_H(\Uu)\right)= p_{\sigma(H)}(\Uu),$$
where the interpretation of $f$ in $\Uu$ is denoted by $f$ as well. Hence for any $h\in p_H(\Uu)$, there is $\sigma_h\in \aut(\Uu)$ extending $\sigma$ and such that
$$\sigma_h(h)=f(h).$$
 Since $(H,f)^{\sharp}$ (see Definition \ref{sharp}) is a quantifier-free definable $L'_{\sigma}$-definable set and $(M,\sigma)$ is existentially closed in $(\Uu,\sigma_h)$ (by genericity of $\sigma)$, we get
 $$(H,f)^{\sharp}(\Uu,\sigma_h)\subseteq H_0(\Uu).$$
 However, $h\in (H,f)^{\sharp}(\Uu,\sigma_h)$ and $h\notin H_0(\Uu)$ (since $h$ is generic over $G$ in the connected group $H(\Uu)$ and $H_0(\Uu)$ is a proper $G$-definable subgroup of $H(\Uu)$), a contradiction.
\end{proof}
\begin{remark}
The result above still captures relatively few of the strong properties of generic automorphisms. In Section~\ref{secaxiom}, we will show that such automorphisms satisfy (at least in the context of finite Morley rank) ``ACFA-like'' axioms. Furthermore, in Section~\ref{secalggps}, we will see that tightness is implied by only a very small part of these axioms.
\end{remark}
Following \cite[Definition 4.1]{pinulla}, we also consider the following strengthening of the notion of tightness.
\begin{definition}\label{defstight}
An automorphism $\alpha$ of a stable group $G$ is called \emph{supertight} (resp. \emph{supertight}$^{\sharp}$) if both of the following hold.
\begin{enumerate}
\item For any $n>0$, $\alpha^n$ is a tight (resp. tight$^{\sharp}$) automorphism of $G$.
\item  For any $m,n>0$, if $m|n$ then $\fix(\alpha^m) \lneqq  \fix(\alpha^n)$.
\end{enumerate}
\end{definition}

We provide below the ``model-theoretic step'' from the introduction to \cite{pinulla}.
\begin{theorem}\label{gensup}
Any generic automorphism is supertight$^{\sharp}$; in particular, it is also supertight.
\end{theorem}
\begin{proof}
Item (1) from Definition \ref{defstight} follows from Theorem \ref{kgen}(1) and Theorem \ref{gentight}. Item (2) follows from Theorem \ref{kgen}(2). The ``in particular'' part follows from Remark \ref{remsharp}(1).
\end{proof}
We obtain that there are plenty of supertight automorphisms.
\begin{corollary}\label{exist}
Let $G$ be a stable group and $\sigma\in \aut(G)$. Then there is an elementary extension $G\preccurlyeq \widetilde{G}$ and $\widetilde{\sigma}\in \aut(\widetilde{G})$ such that $\widetilde{\sigma}$ is supertight and $\widetilde{\sigma}$ extends $\sigma$.

Moreover, $\widetilde{\sigma}$ can be taken as supertight$^{\sharp}$ or generic.
\end{corollary}
\begin{proof}
It follows from Theorem \ref{gensup} and Lemma \ref{ec}.
\end{proof}

\section{Some properties of generic automorphisms}\label{secbetter}
In this section, we discuss several issues related to genericity and tightness. We also answer a question from \cite{pinulla}.

\subsection{Axiomatization}\label{secaxiom}
Let us assume that $T$ is an $\omega$-stable theory. In this subsection, we collect from the literature some results regarding axioms for generic automorphisms. We recall the following definition (see e.g. \cite[Definition 1.1]{KiPi}).
\begin{definition}
The theory $T$ has the \emph{Definable Multiplicity Property}, abbreviated \emph{DMP}, if for any formula $\phi(x;y)\in L$ and any $a,a'\in \Uu$, there is a formula $\theta(y)\in \tp(a)$ such that whenever
$\Uu\models \theta(a')$ then we have:
$$\mathrm{RM}\left(\phi(x;a')\right)=\mathrm{RM}\left(\phi(x;a)\right),\ \ \ \ \mathrm{DM}\left(\phi(x;a')\right)=\mathrm{DM}\left(\phi(x;a)\right),$$
where DM is the Morley degree.
\end{definition}
The following was shown in \cite{KiPi} in the strongly minimal case and then generalized in \cite{BGH}. For the notion of an \emph{almost $\aleph_1$-categorical} theory, we refer the reader to \cite[Definition 4.1]{BGH}. By \cite[Fact 4.2]{BGH}, any almost $\aleph_1$-categorical theory has finite Morley rank.
\begin{theorem}[Corollary 4.11 in \cite{BGH}]\label{dmpgen}
Suppose that $T$ is almost $\aleph_1$-categorical, for example, $T = \mathrm{Th}(G)$ for $G$ a group of
finite Morley rank. Then the class of models of $T$ together with generic automorphisms is elementary if and only if $T$ has DMP.
\end{theorem}
The axioms for the class of models with generic automorphism appearing in Theorem \ref{dmpgen} were specified in \cite{ChPi} (see Item (ii) in the beginning of \cite[3.11]{ChPi}) and we reproduce them below.
\parm
\emph{Whenever $V$ is a set of Morley rank $n$ and Morley degree 1, and $W$ is a definable subset of $V \times \sigma(V)$
of Morley rank $n + m$ and Morley degree 1 such that the projections of $W$ on each of $V$ and $\sigma(V)$
have Morley rank $n$ and all the fibres of each of these projections all Morley rank $m$, then there is $v\in V$ such that $(v, \sigma(v))\in W$.}
\parm
Following the proof of Theorem~\ref{gentight} (or the original argument in \cite[Theorem 1.1]{acfa1}), it is straightforward to show that any generic automorphism satisfies these axioms. It is worth mentioning that this holds even without DMP, which ensures that the axioms are first-order. These axioms provide a strong form of the tightness$^{\sharp}$ condition; to see this, one should take $V := H$ and $W := \operatorname{graph}(f)$ in the setting of Definition~\ref{sharp}.
\subsection{Algebraic groups}\label{secalggps}
Let $K$ be an algebraically closed field and $\sigma\in \aut(K)$ be such that
\begin{itemize}
  \item the field $\fix(\sigma)$ is pseudofinite,

  \item the induced $\bar{\sigma}\in  \aut(\mathrm{PGL}_2(K))$ is supertight.
\end{itemize}
In this subsection, we address the following.
\begin{question}[Question 6.1 in \cite{pinulla}]\label{q}
Assume that $\sigma$ is as above. When $(K,\sigma)\models \mathrm{ACFA}$?
\end{question}
By Remark \ref{present}(1), we get that $(K,\sigma)\models \mathrm{ACFA}$ if and only if $\sigma$ is generic. We will show that the above assumptions are not enough for genericity of $\sigma$ (Corollary \ref{notenough}) and that genericity of $\sigma$ is equivalent to genericity of $\bar{\sigma}$ (Proposition \ref{equiv}).

We first observe that the pseudofiniteness assumption implies the supertightness assumption above. For an algebraic group $H$ over $K$, we denote by
$$\sigma_K:H(K)\to H(K)$$
the induced automorphism of the group of $K$-rational points of $H$.
\begin{proposition}\label{1implies2}
Let $\sigma\in \aut(K)$ such that $\fix(\sigma)$ is a pseudofinite field. Then, the induced automorphism $\bar{\sigma}=\sigma_{\mathrm{PGL}_2}$ on $\mathrm{PGL}_2(K)$ is supertight.
\end{proposition}
\begin{proof}
Let us denote $C:=\fix(\sigma)$. Actually, the proof will work for many connected algebraic groups over $C$ instead of $\mathrm{PGL}_2$ (see Remark \ref{tricky}(2)). Let us recall here that by a result of Ax (see \cite[Corollary 23.10.5]{FrJa}), $C$ is perfect, PAC and the absolute Galois group of $C$ coincides with the profinite completion of $\Zz$. Since $C$ is not algebraically closed, we regard algebraic groups over $C$ as appropriate functors and do not identify them with the groups of rational points. Let us fix an arbitrary connected algebraic group $G$ over $C$.
%We denote by $G(\sigma)$ the induced automorphism of $G(K)$.

\medskip
\textbf{Claim 1}
\\
The automorphism $\sigma_G$ is tight.
\begin{proof}[Proof of Claim 1]
Take any connected $C$-definable subgroup $T\leqslant G(K)$. Since $C$ is perfect, such $T$ coincides with $H(K)$, where $H$ is an algebraic subgroup of $G$ which is defined over $C$ (see e.g. Item 2. at the beginning of \cite[Section 4.5]{Po1}). Since $H$ is algebraic, we have
\begin{equation}
  \tag{$***$}
\fix\left(\sigma_G|_{H(K)}\right)=\fix(\sigma_H)=H(C).
\label{star0}
\end{equation}
(The first equality is obvious, but the second is a bit tricky, see Remark \ref{tricky}(2)). Since $C$ is PAC, $H(C)$ is Zariski dense in $H(K)$ by \cite[Proposition 12.1.1]{FrJa}, which gives tightness using Remark \ref{remtight}(2).
\end{proof}
The next claim is very general and obvious. We will see soon that we get the equality there after using the pseudofiniteness assumption.

\medskip
\textbf{Claim 2}
\\
For all $n>0$ we have
$$[\fix(\sigma^n):C]\leqslant n.$$
\begin{proof}[Proof of Claim 2]
It follows from basic Galois theory, since
$$C=\fix\left(\sigma|_{\fix(\sigma^n)}\right)$$
and $\sigma$ has order at most $n$ after restricting to $\fix(\sigma^n)$.
\end{proof}
Since any finite extension of a pseudofinite field is again pseudofinite, $\sigma_G$ satisfies Item (1) from the definition of supertightness (Definition \ref{defstight}) by Claims 1 and 2.

To show Item (2) from Definition \ref{defstight}, we will use some properties of $\mathrm{PGL}_2$ (however, see Remark \ref{tricky}(2)). We need the following first.

\medskip
\textbf{Claim 3}
\\
For each $n\neq m$, we have $\fix(\sigma^n)\neq \fix(\sigma^m)$.
\begin{proof}[Proof of Claim 3]
It is enough to show that for each $n>0$, we have $[\fix(\sigma^n):C]=n$. By Claim 2, it is enough to show that $[\fix(\sigma^n):C]\geqslant n$. Since $C$ is pseudofinite, it has a Galois extension $C\subseteq C'$ of degree $n$. Then $\sigma|_{C'}\in \gal(C'/C)$ and $|\gal(C'/C)|=n$, so we have
$$\sigma^n|_{C'}=\left(\sigma|_{C'}\right)^n=\id_{C'}.$$
Hence we get $C'\subseteq \fix(\sigma^n)$ and $[\fix(\sigma^n):C]\geqslant n$ indeed.
\end{proof}
Similarly as in $(***)$ above, for any $n>0$ we have
$$\fix\left(\sigma^n_{\mathrm{PGL}_2}\right)=\mathrm{PGL}_2(\fix(\sigma^n)).$$
So, it is enough to show that for $n\neq m$, we have
$$\mathrm{PGL}_2(\fix(\sigma^n))\neq \mathrm{PGL}_2(\fix(\sigma^m)).$$
By Claim 3, the above holds for the additive group $\ga$ in place of $\mathrm{PGL}_2$. Since there is an algebraic subgroup of $\mathrm{PGL}_2$ (consisting of cosets of the upper unitriangular matrices) isomorphic to $\ga$, we can conclude the proof, since all the data here  is defined over $C$ (and even over the prime subfield of $C$).
\end{proof}
\begin{corollary}\label{notenough}
There is $\sigma\in \aut(\Qq^{\alg})$ such that:
\begin{enumerate}
  \item $\fix(\sigma)$ is a pseudofinite field;
  \item $\sigma$ is not generic;
  \item $\sigma_{\mathrm{PGL}_2}\in \aut(\mathrm{PGL}_2(K))$ is supertight.
\end{enumerate}
\end{corollary}
\begin{proof}
By Proposition \ref{1implies2}, it is enough to find $\sigma\in \aut(\Qq^{\alg})$ such that
 $\fix(\sigma)$ is a pseudofinite field and $\sigma$ is not generic. By \cite[Theorem 23.5.1]{FrJa} (for $e=1$), there is $\sigma\in \aut(\Qq^{\alg})$ such that $\fix(\sigma)$ is pseudofinite. However, by \cite[Section 13.3]{HrFro} (the second paragraph of the subsection titled `The examples of Cherlin-Jarden') if $(K,\sigma)\models \mathrm{ACFA}_0$, then $K$ has infinite transcendence degree over $\Qq$.
\end{proof}

\begin{remark}\label{tricky}
We comment here on several issues related to the proof of Proposition \ref{1implies2}.
\begin{enumerate}
  \item We used several times in the proof the fact saying that if $H$ is an algebraic group over $C$, then ``$H$ commutes with Fix'' that is:
   $$\fix(\sigma_H)=H(\fix(\sigma)).$$
   This may seem obvious, but it uses the fact the functor (of rational points) given by an algebraic group is \emph{representable}. It is not true for other functors like $\mathrm{PSL}_2$; for example, if $\sigma\in \aut(\Cc)$ is the complex conjugation, then $\fix(\sigma)=\Rr$, however we have (since $\mathrm{PSL}_2(\Cc)=\mathrm{PGL}_2(\Cc)$):
   $$\fix\left(\sigma_{\mathrm{PSL}_2}\right)=\fix\left(\sigma_{\mathrm{PGL}_2}\right)=\mathrm{PGL}_2(\Rr)\ncong \mathrm{PSL}_2(\Rr).$$

  \item In the proof of Proposition \ref{1implies2}, we did not use much about $\mathrm{PGL}_2$ besides
  \pars
  ``there is an algebraic subgroup of $\mathrm{PGL}_2$ isomorphic to $\ga$ and defined over $C$''.
  \pars
Hence, with suitable modifications, the argument likely extends to an arbitrary algebraic group in place of $\mathrm{PGL}_2$.
\end{enumerate}
\end{remark}

The following result provides a direct answer to \cite[Question 6.1]{pinulla} for all simple algebraic groups over algebraically closed fields.
\begin{proposition}\label{equiv}
Let $G(K)$ be an infinite simple algebraic group over an algebraically closed field $K$ and $\sigma\in \aut(K)$. Then, $\sigma$ is generic if and only if $\sigma_G\in \aut(G(K))$ is generic.
\end{proposition}
\begin{proof}
By \cite{zilbiint} and \cite[Corollary 4.16]{Po1}, the field $K$ and the group $G(K)$ are bi-interpretable, so the result follows directly from Proposition \ref{biint}.
\end{proof}
In particular, we get examples of supertight automorphisms which are not generic.
\begin{remark}
By applying \cite[Corollary 1.2(i)]{segtent}, Proposition~\ref{equiv} can be generalized to arbitrary fields (and potentially beyond using \cite[Theorem 1.1]{segtent}), where the role of the simple algebraic group $G$ is taken by a \emph{Chevalley–Demazure group scheme} over $\mathbb{Z}$ (see \cite{segtent}). Each simple algebraic group over an algebraically closed field arises from a Chevalley-Demazure group scheme and these group schemes are classified e.g. in \cite[Table 9.2]{linlie}.
\end{remark}

\subsection{Pseudofiniteness}\label{secpf}
Let $C=\fix(\sigma)$ where $\sigma\in \aut(M)$ is generic.
We want to discuss the following possible similarities with Ax's characterization of pseudofinite fields (see the beginning of the proof of Proposition \ref{1implies2}).

It is clear that for any $c\in \dcl(C)$, the singleton $\{c\}$ is $\sigma$-invariant, so (since $\sigma\in \aut_{C}(\Uu)$) $\dcl(C)=C$. Therefore, $C$ is a definably closed substructure of $M$ and no assumption on the automorphism $\sigma$ were used here.

Regarding the PAC condition, Hrushovski defines a substructure $F \subseteq M$ to be PAC in the context of a strongly minimal theory (\cite[Definition 1.2]{Hrupfstr}) if every formula over $F$ of Morley degree 1 has a solution in $F$. This definition naturally extends to any $\omega$-stable theory and was further generalized to the stable context in \cite{PiPolk}. PAC substructures of stable structures were also investigated in \cite{HK4}; in particular, by \cite[Lemma 3.7]{HK4}, $C$ is PAC in the sense of \cite[Definition 2.3]{HK4}, a notion which specializes to \cite[Definition 1.2]{Hrupfstr} for the $\omega$-stable case.

Regarding the absolute Galois group case of Ax's description, we should show the following
$$\gal(C):=\aut(\acl(C)/C)=\overline{\langle \sigma|_{\acl(C)}\rangle}\cong \widehat{\Zz}.$$
In the course of the proof, one should show first that
$$\operatorname{Per}(\sigma):=\bigcup_{n>0}\fix(\sigma^n)=\acl(C).$$
Clearly, the inclusion $\acl(C) \subseteq \operatorname{Per}(\sigma)$ always holds; however, the reverse inclusion fails, for instance, in the case of a pure set. To obtain the opposite inclusion, one must assume a version of elimination of imaginaries. The minimal requirement appears to be the coding of finite sets, as in the result below (cf. \cite[Theorem 12]{MedTak}). We recall that $T$ \emph{codes finite sets} if, for any $m > 0$ and any finite $F \subset \Uu^m$, there exists a finite tuple $b$ from $\Uu$ such that for every $\tau \in \aut(\Uu)$, we have $\tau(b) = b$ if and only if $\tau(F) = F$. In this case, $b$ is referred to as the \emph{code} of $F$.

The next result does not use genericity of $\sigma$.
\begin{lemma}\label{eifin}
If $T$ codes finite sets, then we have
$$\acl(C)=\mathrm{Per}(\sigma).$$
\end{lemma}
\begin{proof}
It is enough to show that $\mathrm{Per}(\sigma)\subseteq \acl(C)$. Let $a\in \mathrm{Per}(\sigma)$, so there is $n>0$ such that $\sigma^n(a)=a$. We consider the finite set $F:=\{a,\sigma(a),\ldots,\sigma^{n-1}(a)\}$ and let $b$ be the code of $F$. We have that $\sigma(F)=F$, so (since $T$ codes finite sets) we get $b\in C^k$ for some $k>0$. Therefore $F$ is $C$-definable and $a\in \acl(C)$.
\end{proof}
%If $T'$ eliminates finite tuples, then we have
%$$\fix(\sigma^n)\subseteq \acl(C).$$
%For all $n>0$, we have
%    $$\aut(\fix(\sigma^n)/C)=\Zz/n\Zz=\langle \sigma|_{\fix(\sigma^n)}\rangle.$$
%We always have
%    $$(*)\ \ \ \ \ \langle \sigma|_{\fix(\sigma^n)}\rangle\leqslant \aut(\fix(\sigma^n)/C)$$
%    and using genericity of $\sigma$ we get
%    $$\Zz/n\Zz=\langle \sigma|_{\fix(\sigma^n)}\rangle.$$
% So, the point is how to get the equality in $(*)$.
The result below was stated by Hrushovski in \cite{Hrupfstr} (above Proposition 11.1) within the Shelah's structure $M^{\mathrm{eq}}$ (that is, $M$ together with all the imaginary sorts) and in the context of finite Morley rank with DMP.
\begin{proposition}\label{zhat}
If $T$ codes finite sets and $\sigma$ is generic, then we have
  $$\gal(C)\cong \widehat{\Zz}.$$
\end{proposition}
\begin{proof}
By Theorem \ref{kgen}(2), we have
$$\langle \sigma|_{\fix(\sigma^n)}\rangle\cong \Zz/n\Zz.$$
Hence, by Lemma \ref{eifin}, it is enough to show that
$$\langle \sigma|_{\fix(\sigma^n)}\rangle=\aut(\fix(\sigma^n)/C).$$
Let us take $\tau\in \aut(\fix(\sigma^n)/C)$ and $a\in \fix(\sigma^n)$ such that the elements 
$$a,\sigma(a),\ldots,\sigma^{n-1}(a)$$ 
are pairwise different. Since, as above, the set $\{a,\sigma(a),\ldots,\sigma^{n-1}(a)\}$ is $C$-definable, there is $k\in \{0,1,\ldots,n-1\}$ such that $\tau(a)=\sigma^k(a)$. We will show that
$$\tau=\left(\sigma|_{\fix(\sigma^n)}\right)^k.$$
Assume not and take $b\in \fix(\sigma^n)$ such that $\tau(b)\neq \sigma^k(b)$. As above, the set
$$\{(a,b),\sigma(a,b),\ldots,\sigma^{n-1}(a,b)\}$$
is $C$-definable as well, so there is $l\in \{0,1,\ldots,n-1\}$ such that $\tau(a,b)=\sigma^l(a,b)$. Then we have
$$\sigma^l(a)=\tau(a)=\sigma^k(a),$$
hence, since the elements $a,\sigma(a),\ldots,\sigma^{n-1}(a)$ are pairwise different, we get $k=l$. Therefore, we have
$$\sigma^k(b)=\sigma^l(b)=\tau(b)\neq \sigma^k(b),$$
a contradiction.
\end{proof}
It is not clear for us whether simple groups of finite Morley rank code finite sets. Given that vector spaces fail to code finite sets, the assumption of simplicity has to be used indeed. Nevertheless, the following result suggests that coding finite sets is a natural assumption when working toward the Cherlin–Zilber conjecture.
\begin{proposition}\label{ei}
Let $G(K)$ is an infinite simple algebraic group over an algebraically closed field $K$. Then $G(K)$ has elimination of imagineries in the language of groups. In particular, $G(K)$ codes finite sets.
\end{proposition}
\begin{proof}
Since the theory of algebraically closed fields have elimination of imagineries, it is enough to observe that the bi-interpretation between the pure group $G(K)$ and the field $K$ (discussed around Proposition \ref{equiv}) does not use any imaginary sorts. Since $K$ is algebraically closed, we can use Zilber's description from \cite[Lemma 2]{zilbiint}. The group $B$ appearing there is a Borel subgroup of $G(K)$, so it is definable in the home sort of $G(K)$. The field $K$ is defined on a minimal normal definable subgroup of $B$, so it lives in the home sort again. Clearly, the algebraic group $G(K)$ lives in the home sort of $K$ as well as a subgroup of the group of invertible matrices over $K$ of some fixed size.
\end{proof}

\begin{remark}
One can identify model-theoretic properties satisfied by infinite simple algebraic groups over algebraically closed fields and subsequently test the Cherlin–Zilber conjecture by investigating whether simple groups of finite Morley rank share these properties. We list several such properties below. Let $G$ be an infinite simple algebraic group over an algebraically closed field:
\begin{enumerate}
\item The fixed-point subgroup of a generic automorphism of $G$ is pseudofinite (this property is crucial!).
\item The group $G$ eliminates imaginaries (see Proposition~\ref{ei}).
\item The group $G$ is model complete. (This was recently established in \cite{HKTY} for any simple algebraic group over an arbitrary model complete field.)\end{enumerate}
\end{remark}
After all of this, the most significant obstacle remains providing an ``Ax-like'' characterization of pseudofinite groups. It is not true in general that if $\sigma$ is a generic automorphism of $M$, then $\fix(\sigma)$ is a pseudofinite model of $T = \operatorname{Th}(M)$. A counterexample is provided by any differentially closed field $M$.

Nevertheless, in \cite[Proposition 11.1]{Hrupfstr}, Hrushovski establishes the existence of a definable measure in the setting of Proposition~\ref{zhat}. He suggests that certain arguments from finite group theory may be applicable to definable sets in $\fix(\sigma)$, potentially providing a path toward proving the algebraicity of $G$ by leveraging Galois theory and numerical invariants beyond pure dimension theory.
\bibliographystyle{plain}
\bibliography{harvard}

\end{document}